\documentclass[leqno]{amsart}
\pdfoutput=1
\usepackage[whole]{bxcjkjatype}
\usepackage[a4paper, margin=3cm]{geometry}
\usepackage{amsthm}
\usepackage{amsmath,amssymb,mathtools,mathdots}
\usepackage{mathrsfs}
\usepackage{url}
\usepackage[all]{xy}
\usepackage{tikz}
\usepackage{tikz-cd}
\usetikzlibrary {arrows,backgrounds}
\usetikzlibrary{shapes.geometric, intersections, calc}
\usetikzlibrary{decorations.pathmorphing, arrows}
\usepackage{comment}
\usepackage{placeins}

\usepackage{multirow}
\usepackage{diagbox}
\usepackage{makecell}
\usepackage{caption}

\newtheorem{theorem}{Theorem}[section]
\newtheorem*{theorem*}{Theorem}
\newtheorem{lemma}[theorem]{Lemma}
\newtheorem*{lemma*}{Lemma}
\newtheorem{proposition}[theorem]{Proposition}
\newtheorem*{proposition*}{Proposition}

\newtheorem*{corollary*}{Corollary}

\newtheorem*{claim*}{Claim}

\newtheorem*{fact*}{Fact}

\newtheorem*{conjecture*}{Conjecture}
\theoremstyle{definition}
\newtheorem{definition}[theorem]{Definition}
\newtheorem*{definition*}{Definition}
\newtheorem{example}[theorem]{Example}
\newtheorem*{example*}{Example}
\newtheorem{remark}[theorem]{Remark}
\newtheorem*{remark*}{Remark}
\newtheorem{question}[theorem]{Question}
\newtheorem*{question*}{Question}

\newtheorem*{assumption*}{Assumption}
\numberwithin{equation}{section}
\allowdisplaybreaks[1]

\DeclareMathOperator{\rank}{rank}

\DeclareMathOperator{\ad}{ad}

\DeclareMathOperator{\Hom}{Hom}

\DeclareMathOperator{\id}{id}

\DeclareMathOperator{\Span}{Span}
\DeclareMathOperator{\BCH}{BCH}

\newcommand{\C}{{\mathbb C}}
\newcommand{\R}{{\mathbb R}}

\newcommand{\Z}{{\mathbb Z}}

\newcommand{\pa}{{\partial}}
\newcommand{\bpa}{{\overline{\partial}}}
\newcommand{\iu}{{\mathrm{i}}}

\newcommand{\conj}[1]{{\overline{#1}}}

\newcommand{\cz}{{\overline{z}}}
\newcommand{\cZ}{{\overline{Z}}}
\newcommand{\cw}{{\overline{w}}}

\newcommand{\fg}{{\mathfrak g}}
\newcommand{\fh}{{\mathfrak h}}

\newcommand{\fk}{{\mathfrak k}}

\usepackage{multirow}

\address{Graduate School of Mathematical Sciences, 
The University of Tokyo,
3-8-1 Komaba, Meguro-ku, 
Tokyo 153-8914, Japan}
\email{shuho@ms.u-tokyo.ac.jp}
\begin{document}
\title
[
Holomorphic polynomial crystallographic actions of nilpotent groups
]
{
Holomorphic polynomial crystallographic actions of nilpotent groups
}
\author{Shuho Kanda}
\date{}
\maketitle
\begin{abstract}
    We prove that every simply connected nilpotent Lie group endowed with a
left-invariant nilpotent complex structure is biholomorphic to $\C^n$.
Moreover, we construct such a biholomorphism explicitly by polynomial maps
in exponential coordinates.  As a consequence, every lattice in such a Lie
group admits a free, properly discontinuous and cocompact action on $\C^n$ by
holomorphic polynomial automorphisms.  We interpret this consequence as a
holomorphic analogue of polynomial crystallographic actions 
introduced by Dekimpe, Igodt, and Lee.
\end{abstract}

\setcounter{tocdepth}{1}
\tableofcontents

\section{Introduction}

\subsection{A motivating example and the questions}

The main objects studied in this paper are simply connected nilpotent Lie groups $G$
endowed with left-invariant complex structures $J$.
The simplest non-abelian example appears in the Kodaira--Thurston manifold.

Let $H_3$ denote the Heisenberg group (see Section \ref{section:Examples}).
The Kodaira--Thurston manifold $X$ is obtained by equipping
$G=H_3(\R)\times \R$ with a left-invariant complex structure $J$ 
and taking the quotient by the lattice
$\Gamma=H_3(\Z)\times \Z$:
\[
    X=(\Gamma \backslash G,J).
\]
This is a classical example of a compact complex surface which admits no K\"ahler metric.
On the other hand, the same manifold can also be constructed as the quotient of $\C^2$
by the following affine, and in particular polynomial, action of $\Gamma$:
\[
    \left\{
    \begin{aligned}
        \gamma_1 \cdot (z,w)&=(z+\iu,w+z),\\
        \gamma_2 \cdot (z,w)&=(z+1,w),\\
        \gamma_3 \cdot (z,w)&=(z,w+\iu),\\
        \gamma_4 \cdot (z,w)&=(z,w+1).
    \end{aligned}
    \right.
\]
Here $\{ \gamma_i \}_{i=1}^4$ are generators of $\Gamma$.
The compatibility of these two descriptions shows that the universal cover of $X$
can be identified biholomorphically with $\C^2$, and that the deck transformations
are given by holomorphic polynomial automorphisms under this identification.

This raises the question of whether the same phenomenon holds for more general pairs $(G,J)$.
More precisely, we are led to the following questions: 

\begin{question}[\cite{Has10}, Conjecture (i)]
\label{Que:question1}
    Let $G$ be a simply connected nilpotent Lie group of real dimension $2n$,
    and let $J$ be a left-invariant complex structure on $G$.
    Is $(G,J)$ biholomorphic to $\C^n$?
\end{question}

\begin{question}\label{Que:question2}
    Let $G$ be a simply connected nilpotent Lie group of real dimension $2n$,
    and let $J$ be a left-invariant complex structure on $G$.
    Suppose that $(G,J)$ is biholomorphic to $\C^n$.
    Does there exist a biholomorphism
    \[
        \Phi \colon (G,J) \xrightarrow{\ \sim\ } \C^n
    \]
    such that, for every $g\in G$, the automorphism
    \[
        \Phi\circ L_g\circ \Phi^{-1} 
        \, \colon \, \C^n \xrightarrow{\ \sim\ } \C^n, 
    \]
    where $L_g$ denotes the left translation by $g$, 
    is given by holomorphic polynomial functions?
\end{question}

A weaker, 
and perhaps more natural, 
lattice version of this question is the following: 

\begin{question}[$\Leftrightarrow$ Question 
\ref{que:main_question}]\label{Que:question2_lattice}
    Let $G$ be a simply connected nilpotent Lie group of real dimension $2n$, 
    and let $\Gamma \subset G$ be a lattice. 
    Does $\Gamma$ admit an action on $\C^n$ by holomorphic polynomial automorphisms
    which is free, properly discontinuous and cocompact?
\end{question}

If $(G,J)$ satisfies the property stated in Question \ref{Que:question2}, 
then the desired action in 
Question \ref{Que:question2_lattice} is obtained by setting 
the action of $\Gamma$ on $\C^n$ by 
$\gamma \cdot z \coloneqq \Phi\circ L_{\gamma}\circ \Phi^{-1}(z)$. 
Conversely, one may also wonder whether an action of $\Gamma$ as in 
Question \ref{Que:question2_lattice} should come from a picture like the one in 
Question \ref{Que:question2}.

We also note that, 
in the special case where $G$ is a complex nilpotent Lie group 
and $J$ is induced by its complex structure, 
the property stated in Question \ref{Que:question2} holds in an elementary way. 
Indeed, 
the exponential map identifies $G$ biholomorphically with $\C^n$, 
and left translations are polynomial in the corresponding coordinates 
by the Baker--Campbell--Hausdorff formula. 

\subsection{Background for the two questions}
Question \ref{Que:question1} was explicitly 
formulated as a conjecture by Hasegawa \cite{Has10}. 
To the best of the author's knowledge, 
there has been little progress on this question. 
For solvable Lie groups, 
the analogous statement is false: as in the case of Inoue surfaces, 
the universal cover may be biholomorphic to 
$\C \times \mathbb{H}$ rather than to $\C^2$. 
In this broader solvable setting, 
however, it remains open whether the universal cover is always Stein. 
Even in the nilpotent case, the Steinness of $(G,J)$ is still unknown in general. 
The results of this paper answer Question \ref{Que:question1} 
affirmatively for a certain class of left-invariant complex structures $J$.

Question \ref{Que:question2} has a deep background 
going back to Milnor's 1977 paper \cite{Mil77}. 
The relevant class of groups in this context is the class of 
\emph{polycyclic-by-finite} groups, 
which contains the lattices in nilpotent Lie groups considered here. 
This background is explained in detail, for example, in 
\cite{Dek96}, \cite{DI97b}, and \cite{Dek00}: 
we briefly recall it here.

In 1977, Milnor asked whether every 
torsion-free polycyclic-by-finite group $\Gamma$ 
occurs as the fundamental group of a compact complete affinely flat manifold. 
Equivalently, this asks whether $\Gamma$ admits a free, 
properly discontinuous, and cocompact affine action on $\R^m$, 
where $m$ is the Hirsch length of $\Gamma$. 
Such an action of $\Gamma$ is called an \emph{affine crystallographic action} 
\cite{FG83}, \cite{GS94}. 
This question remained open for a long time, 
until Y. Benoist found counterexamples among nilpotent groups 
\cite{Ben92}, \cite{Ben95}. 
Later, 
Burde and Grunewald found simpler counterexamples 
\cite{BG95}, \cite{Bur96}. 

Inspired by this problem, 
Dekimpe, Igodt, and Lee introduced 
the notion of a 
\emph{polynomial crystallographic action} in \cite{DIL96}. 
This is a polynomial analogue of an affine crystallographic action, 
obtained by replacing affine actions with polynomial actions: 
see Definition \ref{Def:real polynomial crystallographic action}. 
When $\Gamma$ is nilpotent, 
it is a lattice in the nilpotent Lie group $G_{\Gamma}$, 
called the real Malcev completion of $\Gamma$, 
which is diffeomorphic to $\R^m$ through exponential coordinates. 
Thus the action of $\Gamma$ on $G_{\Gamma}$ by left translations gives, 
via the Baker--Campbell--Hausdorff formula, 
a polynomial action on $\R^m$. 
Dekimpe and Igodt proved in \cite{DI97a}, \cite{DI97b} 
that this existence result extends 
to polycyclic-by-finite groups. 
Subsequent work has studied, among other things, 
bounds on the degrees of such polynomial actions.

In this paper, 
we restrict our attention to nilpotent groups and consider 
a holomorphic analogue of these notions. 
Here, 
$P_h(\C^n)$ denotes the group of holomorphic 
polynomial automorphisms of $\C^n$.

\begin{definition}
[= Definition \ref{Def:holomorphic polynomial crystallographic action}]
Let $\Gamma$ be a torsion-free finitely generated nilpotent 
group with $h(\Gamma)=2n$. 
A \emph{holomorphic polynomial
crystallographic action} of $\Gamma$ is a homomorphism 
\[
  \rho:\Gamma \to P_h(\C^n)
\]
such that the induced action of $\Gamma$ on $\C^n$ is free, properly
discontinuous, and cocompact. 
\end{definition}

Question \ref{Que:question2_lattice} asks 
whether every torsion-free finitely generated nilpotent group 
$\Gamma$ with even Hirsch length admits 
a holomorphic polynomial crystallographic action. 
At present, no counterexample to this question seems to be known. 
The results of this paper may be viewed as giving an affirmative answer 
for a certain class of such groups $\Gamma$.

\subsection{Results}

Let $G$ be a simply connected nilpotent Lie group of dimension $2n$, and let
$\fg$ be its Lie algebra. 
A left-invariant complex structure $J$ on $G$ is 
equivalent to an integrable complex structure, also denoted by $J$, on $\fg$. 
We recall the notion of a nilpotent complex structure.

\begin{definition}[= Definition \ref{Def:nilpotent complex structure}]
    A complex structure $J$ on a nilpotent Lie algebra $\fg$ 
    is called \emph{nilpotent} if there exists a complex basis 
    $\{Z_i\}_{i=1}^{n}$ of $\fg^{1,0}$ such that 
    \begin{equation}\label{Equ:J_nilpotent_intro}
        [Z_i,Z_j], [Z_i,\conj{Z}_j] \in
        \Span_{\C}(\{Z_{j+1},\conj{Z}_{j+1},\ldots,Z_n,\conj{Z}_n\})
    \end{equation}
    for all $1 \le i\le j \le n$, where $\fg^{1,0}$ denotes the $\iu$-eigenspace of
    $J$ in $\fg \otimes \C$.
\end{definition}

Nilpotent complex structures were introduced in \cite{CFGU97}
in the study of compact nilmanifolds. 
In low dimensions, 
this condition is satisfied by many complex structures on 
nilpotent Lie algebras. 
See Remark \ref{Rem:nilpotent} for further details.

The main theorem of this paper is the following.

\begin{theorem}
[= Theorem \ref{Thm:main_theorem}]
\label{Thm:intro1}
    Let $G$ be a simply connected nilpotent Lie group of dimension $2n$ 
    with a left-invariant nilpotent 
    complex structure $J$. 
    Then there exists a biholomorphic map $\Phi \colon (G,J) \to \C^n$. 
    Furthermore, $\Phi$ and $\Phi^{-1}$ can be taken to be polynomial 
    in the exponential coordinates on $G$.
\end{theorem}

We now reinterpret this consequence from the point of view of 
the lattice $\Gamma$, 
rather than the Lie group $G$. 
Let $\Gamma$ be a torsion-free finitely generated nilpotent group with $h(\Gamma)=2n$. 
By Malcev's theorem, 
there exists a uniquely determined simply connected nilpotent Lie group
$G_{\Gamma}$ in which $\Gamma$ embeds as a lattice. 
We say that $\Gamma$ is of \emph{nilpotent-complex type} 
if $G_{\Gamma}$ admits a left-invariant nilpotent complex structure. 
As an immediate consequence of Theorem \ref{Thm:intro1}, we obtain the following: 

\begin{theorem}
[= Theorem \ref{Thm:main_theorem_lattice}]
\label{Thm:intro2}
Let $\Gamma$ be a torsion-free finitely generated nilpotent group with
$h(\Gamma)=2n$.  If $\Gamma$ is of nilpotent-complex type, then $\Gamma$
admits a holomorphic polynomial
crystallographic action. 
\end{theorem}

At present, 
no counterexample to Question \ref{Que:question2_lattice} seems to be known. 
Thus it is completely unclear whether 
the converse of the preceding theorem should hold. 
As a possible necessary condition for the existence of 
a holomorphic polynomial crystallographic action, 
one might at least expect that the 
Malcev completion $G_{\Gamma}$ admits a left-invariant complex structure, 
but even this is not known. 

On the other hand, 
under an additional boundedness assumption, 
this expectation does hold. 
More precisely, 
in \cite{BD02}, 
Benoist and Dekimpe proved a uniqueness theorem for 
polynomial crystallographic actions of bounded degree,  
that is, 
for polynomial crystallographic actions 
whose degrees are uniformly bounded. 
Using this result, 
we show that if $\Gamma$ admits a 
holomorphic polynomial crystallographic action of bounded degree, 
then its Malcev completion $G_{\Gamma}$ 
admits a left-invariant complex structure.

\begin{theorem}
[= Theorem \ref{Thm:bounded-degree-necessary-condition}]
\label{Thm:intro3}
Let $\Gamma$ be a torsion-free finitely generated nilpotent group with
$h(\Gamma)=2n$.  If $\Gamma$ admits a holomorphic polynomial
crystallographic action of bounded degree, then $G_{\Gamma}$ admits a
left-invariant complex structure $J$ such that
$(G_{\Gamma},J)\simeq \C^n$.
\end{theorem}

\subsection{Organization of the paper}
The paper is organized as follows:
\begin{itemize}
\item In Section~\ref{section:Coordinates and nilpotent complex structures}, 
we explain in detail the construction of exponential coordinates on nilpotent Lie groups. 
After carefully describing 
how left-invariant vector fields are expressed in these coordinates, 
we introduce the notion of a nilpotent complex structure.

\item In Section~\ref{section:Proof of the Main Theorem}, 
we prove Theorem~\ref{Thm:intro1}.

\item In Section~\ref{section:Holomorphic polynomial crystallographic action}, 
we introduce several notions in order to 
relate the theory of polynomial crystallographic actions 
to Theorem~\ref{Thm:intro1}, and then state Theorem~\ref{Thm:intro2}. 
We also prove Theorem~\ref{Thm:intro3} using 
the uniqueness theorem due to Benoist and Dekimpe.

\item In Section~\ref{section:Examples}, 
we explicitly compute, in two examples, the biholomorphisms constructed in 
Theorem~\ref{Thm:intro1} and the holomorphic polynomial crystallographic actions constructed in 
Theorem~\ref{Thm:intro2}. 
These examples may also help the reader follow the proof of Theorem~\ref{Thm:intro1}.
\end{itemize}

\addtocontents{toc}{\protect\setcounter{tocdepth}{0}}
\section*{Acknowledgments}
\addtocontents{toc}{\protect\setcounter{tocdepth}{1}}
I am deeply grateful to Hisashi Kasuya. 
Originally, 
I was only considering Question \ref{Que:question1}. 
When I told him that the biholomorphism could be chosen to be polynomial,  
he pointed out the connection with polynomial crystallographic actions 
in the sense of Dekimpe. 
This perspective gave the present paper a much richer background and motivation. 
This work was supported by JSPS KAKENHI Grant Number 24KJ0931.

\section{Coordinates and nilpotent complex structures}
\label{section:Coordinates and nilpotent complex structures}

\subsection{Global coordinates on simply connected nilpotent Lie groups}

Let $\fg$ be a nilpotent Lie algebra of real dimension $d$, 
and let $G$ be a corresponding simply connected nilpotent Lie group. 
Recall that 
the exponential map $\exp \colon \fg \to G$ is a diffeomorphism. 
Through the exponential map, 
the group law on $G$ 
can be described on $\fg$ by the Baker–Campbell–Hausdorff formula:
\[
    \begin{array}{c}
        \exp(X) \cdot \exp(Y) = \exp(\BCH(X,Y)), \\
        \BCH(X,Y) \coloneqq 
        X+Y+\frac{1}{2}[X,Y]+\frac{1}{12}([X,[X,Y]]+[Y,[Y,X]])+\cdots, 
    \end{array}
\]
for $X,Y \in \fg$. 
The first variation of this formula with respect to $Y$ is given by
\begin{align*}
    \left.\frac{d}{dt}\right|_{t=0}\BCH(X,tY) &= 
    Y+\frac{1}{2}[X,Y]+\frac{1}{12}[X,[X,Y]]
        -\frac{1}{720}[X,[X,[X,[X,Y]]]]+\cdots \\
    &=\frac{\ad_X}{1-e^{-\ad_X}} \cdot Y .
\end{align*}

Let $\mathcal{B} = (X_1, \ldots, X_d)$ be a basis of $\fg$.
We identify $\fg$ with $\R^d$ via
\[
    \psi_{\mathcal{B}} \colon \fg \xrightarrow{\ \sim\ } \R^{d}, \quad
    x_1 X_1+\cdots +x_d X_d \mapsto
    (x_1,\ldots,x_d). 
\]
Define $\kappa \coloneqq \psi_{\mathcal B} \circ \exp^{-1} 
\colon G \xrightarrow{\ \sim\ } \R^{d}$, 
which makes the following diagram commute:
\[
    \begin{tikzcd}[row sep=1.2em, column sep=2.0em]
    & \fg \arrow[dl, "\exp"'] \arrow[dr, "\psi_{\mathcal B}"] & \\
    G \arrow[rr, "\kappa"] & & \R^{d}. 
    \end{tikzcd}
\]
The map $\kappa$ gives global coordinates on $G$, 
which we call the exponential coordinates associated with $\mathcal B$.

For $V \in \fg$, let $\widetilde{V}$ 
denote the corresponding left-invariant vector field on $G$. 
We now express $\widetilde{V}$ in these global coordinates. 
Let $e$ be the identity element of $G$. 
Identifying $\fg$ with $T_e G$, for $V \in \fg$ 
the corresponding left-invariant vector field $\widetilde{V}$ 
is given at $p \in G$ by
\[
    (\widetilde{V})_p = (dL_p)_e (V) \in T_p G,
\]
where $L_p$ denotes the left translation by $p$. 
Setting $p=\exp(P)$, 
we have
\begin{align*}
    (d\exp^{-1})_p\bigl((\widetilde{V})_p\bigr) &= 
    (d\exp^{-1})_p \left( (dL_p)_e \left(\left.\frac{d}{dt}\right|_{t=0} \exp(tV) \right) \right) \\
    &= \left.\frac{d}{dt}\right|_{t=0} \exp^{-1}(L_p(\exp(tV))) \\
    &=\left.\frac{d}{dt}\right|_{t=0} \BCH(P,tV)\\
    &=V+\frac{1}{2}[P,V]+\frac{1}{12}[P,[P,V]]
    -\frac{1}{720}[P,[P,[P,[P,V]]]]+\cdots. 
\end{align*}
Strictly speaking, the terms on the left-hand side lie in $T_P \fg$, 
while the last expression lies in $\fg$. 
In the last equality we implicitly use the canonical identification
\[
    T_P \fg \xrightarrow{\ \sim\ } \fg,
    \qquad
    \left.\frac{d}{dt}\right|_{t=0} (P + tU)
    \mapsto U.
\]
Similarly, we have the canonical identification
\[
    T_{\psi_{\mathcal B}(P)} \R^{d} \xrightarrow{\ \sim\ }\R^{d},
    \quad
    a_1 \pa_{x_1}+\cdots+
    a_d \pa_{x_d} \mapsto (a_1,\ldots,a_d), 
\]
and this identification is compatible with $\psi_{\mathcal B}$, 
that is, 
the following diagram is commutative: 
\[
    \begin{tikzcd}
        T_P \fg \arrow[r, "\sim"] \arrow[d, "(d\psi_{\mathcal{B}})_P"'] 
        & \fg \arrow[d, "\psi_{\mathcal{B}}"] \\
        T_{\psi_{\mathcal B}(P)} \R^{d} \arrow[r, "\sim"] 
        & \R^{d}. 
    \end{tikzcd}
\]
Thus, $\widetilde{V}$ can be computed on $\R^{d}$ as follows:

\begin{proposition}\label{Prop:calculus_of_vector}
    Let $\widetilde{V}$ denote the left-invariant vector field on $G$
    corresponding to $V \in \fg$.
    Let
    \[
    \psi_{\mathcal B} \colon \fg \xrightarrow{\ \sim\ } \R^{d}
    \]
    be the identification determined by the basis
    $\mathcal B \coloneqq (X_1, \ldots, X_d)$.
    
    Write
    \[
        \bigl(d\psi_{\mathcal B} \circ d\exp^{-1}(\widetilde{V})\bigr)_{(x_1,\ldots,x_d)}
        =
        a_1 \pa_{x_1}
        + \cdots
        + a_d \pa_{x_d}.
    \]
    Then
    \begin{align*}
        (a_1,\ldots,a_d)
        &=
        \psi_{\mathcal B}\left(
        V + \frac{1}{2}[P,V]
        + \frac{1}{12}[P,[P,V]]
        - \frac{1}{720}[P,[P,[P,[P,V]]]]
        + \cdots
        \right) \\
        &= \psi_{\mathcal B}\left(
        \frac{\ad_P}{1-e^{-\ad_P}} \cdot V
        \right), 
    \end{align*}
    where $P=x_1X_1+\cdots+x_dX_d$. 
\end{proposition}

\begin{example}
    Let $\fh_3$ be the Heisenberg Lie algebra of dimension $3$, 
    that is, 
    there is a basis $\mathcal{B} \coloneqq (X_1,X_2,X_3)$ of $\fh_3$ 
    such that $[X_1,X_2]=X_3$ and all other brackets vanish. 
    We identify the corresponding simply connected Lie group $H_3$
    with $\R^3$ via global coordinates induced by the exponential map
    and the identification $\psi_{\mathcal B} \colon \fh_3 \xrightarrow{\ \sim\ } \R^3$. 
    Now, since
    \[
        \left\{
        \begin{aligned}
        &X_1 + \frac{1}{2}[P,X_1]=X_1-\frac{1}{2}x_2X_3 \\
        &X_2 + \frac{1}{2}[P,X_2]=X_2+\frac{1}{2}x_1X_3 \\
        &X_3 + \frac{1}{2}[P,X_3]=X_3, 
        \end{aligned}
        \right. 
    \]
    where $P=x_1X_1+x_2X_2+x_3X_3$, 
    the left-invariant vector fields $\widetilde{X}_i$ on $H_3$
    are expressed in these coordinates as
    \[
        \left\{
        \begin{aligned}
        &\widetilde{X}_1=\pa_{x_1}-\frac{1}{2}x_2 \pa_{x_3} \\
        &\widetilde{X}_2=\pa_{x_2}+\frac{1}{2}x_1 \pa_{x_3} \\
        &\widetilde{X}_3=\pa_{x_3}.  
        \end{aligned}
        \right. 
    \]
\end{example}

\subsection{Global complex coordinates on simply connected nilpotent Lie groups 
with left-invariant complex structures}
\label{subsection:Global complex coordinates}

Let $\fg$ be a nilpotent Lie algebra of real dimension $2n$. 
An \emph{almost-complex structure} $J$ on $\fg$ is a linear map 
$J \colon \fg \to \fg$ such that $J^2 = -\id_{\fg}$. 
Let $\fg_{\C} \coloneqq \fg \otimes \C$ be the complexification of $\fg$. 
We decompose $\fg_{\C}$ into the eigenspaces of $J$, 
and denote by $\fg^{1,0}$ and $\fg^{0,1}$ 
the eigenspaces corresponding to the eigenvalues $\iu$ and $-\iu$, respectively. 
The almost-complex structure $J$ is said to be \emph{integrable} if 
$\fg^{1,0}$ is a Lie subalgebra of $\fg_{\C}$. 
In this case, 
$J$ is simply called a \emph{complex structure}. 

Let $G$ be a simply connected nilpotent Lie group associated to $\fg$. 
A complex structure $J$ on $\fg$ induces a left-invariant complex structure on $G$, which is also denoted by $J$. 
In this way, 
$(G,J)$ is a simply connected complex manifold of complex dimension $n$.

Let $\{Z_i\}_{i=1}^{n}$ be a basis of $\fg^{1,0}$ and write
$Z_i=(X_i-\iu Y_i)/2$. 
Then the basis $\mathcal{B}=(X_1,Y_1,\ldots,X_n,Y_n)$ induces
global coordinates on $G$ via the map
$\kappa \colon G \to \R^{2n}$. 
By identifying $\R^{2n}$ naturally with $\C^n$, 
we obtain complex coordinates on $G$. 
(Note that these coordinates are not compatible with the complex structure $(G,J)$.) 
We now compute the expression of $\widetilde{V}$ 
in these complex coordinates for $V \in \fg_{\C}$. 
Taking into account that 
$x_i X_i + y_i Y_i = z_i Z_i + \overline{z}_i\,\overline{Z}_i$, 
and applying Proposition \ref{Prop:calculus_of_vector}, 
we obtain the following expression for $\widetilde{V}$ 
in these coordinates: 

\begin{align*}
    \bigl(d\psi_{\mathcal B_{\C}} \circ d\exp^{-1}(\widetilde{V})\bigr)_{(z_1,\ldots,z_n)}
    &=
    a_1 \pa_{z_1} +b_1 \pa_{\conj{z}_1}
    + \cdots
    + a_n \pa_{z_n} +b_n \pa_{\conj{z}_n}, \\
    (a_1,b_1, \ldots,a_n, b_n)
    &=
    \psi_{\mathcal B_{\C}}\left(
    V + \frac{1}{2}[P,V]
    + \frac{1}{12}[P,[P,V]]
    - \frac{1}{720}[P,[P,[P,[P,V]]]]
    + \cdots
    \right), 
\end{align*}
where
$\psi_{\mathcal B_{\C}} \colon \fg_{\C} \xrightarrow{\ \sim\ } \C^{2n}$
is the identification determined by the basis
$\mathcal B_{\C} \coloneqq (Z_1, \conj{Z}_1, \ldots, Z_n, \conj{Z}_n)$ 
and 
$P=z_1Z_1+\conj{z}_1\conj{Z}_1+\cdots+z_nZ_n+\conj{z}_n\conj{Z}_n$. 
See Examples \ref{Ex:2step} and \ref{Ex:3step} 
for explicit computations illustrating this procedure.

\subsection{Nilpotent complex structures}

Let $J$ be a complex structure on a nilpotent Lie algebra $\fg$. 
We define $J$ to be \emph{nilpotent} as follows: 

\begin{definition}\label{Def:nilpotent complex structure}
    A complex structure $J$ on a nilpotent Lie algebra $\fg$ 
    is called \emph{nilpotent} if 
    there exists a complex basis $\{Z_i\}_{i=1}^{n}$ of $\fg^{1,0}$
    such that 
    \begin{equation}\label{Equ:J_nilpotent}
        [Z_i,Z_j], [Z_i,\conj{Z}_j] \in
        \Span_{\C}(\{Z_{j+1},\conj{Z}_{j+1},\ldots,Z_n,\conj{Z}_n\})
    \end{equation}
    for all $1 \le i\le j \le n$. 
\end{definition}

\begin{remark}\label{Rem:nilpotent}
    Nilpotent complex structures were introduced in \cite{CFGU97}
    in the study of compact nilmanifolds, 
    and were further studied in \cite{CFGU99}, \cite{CFGU00}. 
    They can also be defined without choosing a basis, for example in terms of a
    $J$-compatible ascending series. 
    The complex structure underlying a nilpotent complex Lie algebra 
    is nilpotent in this sense. 
    This condition is quite common in low dimensions: 
    in real dimension $4$, 
    every complex structure on a nilpotent Lie algebra is nilpotent, 
    while in real dimension $6$, 
    $18$ of the $34$ non-isomorphic nilpotent Lie algebras admit 
    complex structures, 
    and $16$ non-isomorphic classes admit nilpotent complex structures. 
    See \cite{COUV16} and the references therein.
\end{remark}

For a basis $\{Z_i\}_{i=1}^{n}$ of $\fg^{1,0}$ satisfying (\ref{Equ:J_nilpotent}), 
we compute the expression of $\widetilde{\conj{Z}}_i$ in complex coordinates 
for $1\le i \le n$. 
Let $P=z_1Z_1+\conj{z}_1\conj{Z}_1+\cdots+z_nZ_n+\conj{z}_n\conj{Z}_n$. 
From the expression
\[
    \conj{Z}_i + \frac{1}{2}[P,\conj{Z}_i]
    + \frac{1}{12}[P,[P,\conj{Z}_i]]
    - \frac{1}{720}[P,[P,[P,[P,\conj{Z}_i]]]]
    + \cdots 
\]
and (\ref{Equ:J_nilpotent}), 
we see that $\widetilde{\conj{Z}}_i$ can be expressed as 
\begin{equation}\label{Equ:cZ_i}
    \widetilde{\cZ}_i = \pa_{\cz_i}+ 
    \sum_{j=i+1}^n \left(A_{ij} \pa_{z_j}
    + B_{ij} \pa_{\cz_j} \right), 
\end{equation}
where $A_{ij}, B_{ij}$ are polynomials in 
$\C[z_1,\ldots,z_{j-1},\cz_1,\ldots,\cz_{j-1}]$ with no constant term. 
In particular, 
$\widetilde{\conj{Z}}_n=\pa_{\cz_n}$. 

\section{Proof of the Main Theorem}
\label{section:Proof of the Main Theorem}

In this section, 
we show the following theorem:

\begin{theorem}\label{Thm:main_theorem}
    Let $G$ be a simply connected nilpotent Lie group of dimension $2n$ 
    with a left-invariant nilpotent 
    complex structure $J$. 
    Then there exists a biholomorphic map $\Phi \colon (G,J) \to \C^n$. 
    Furthermore, $\Phi$ and $\Phi^{-1}$ can be taken to be polynomial 
    in the exponential coordinates on $G$.
\end{theorem}

Before proving the theorem, 
we establish the following lemma, 
which may be regarded as a polynomial version of the $
\bpa$-Poincar\'{e} lemma.

\begin{lemma}\label{Lem:Poincare}
    Let $m$ be a positive integer, 
    and denote by $w=(w_1,\ldots,w_m)$ the coordinates on $\C^m$. 
    Suppose that
    \[
    \alpha = \sum_{i=1}^m F_i(w,\conj w) d \conj w_i, 
    \quad F_i \in \C[w_1,\ldots,w_m,\conj w_1, \ldots , \conj w_m] 
    \]
    is a $\bpa$-closed $(0,1)$-form. 
    Then there exists a polynomial
    $G \in \C[w_1,\ldots,w_m,\conj w_1, \ldots , \conj w_m]$
    such that $\bpa G =\alpha$.
\end{lemma}

\begin{proof}
    As in the usual Poincar\'{e} lemma, set
    \[
    G(w,\conj w) = \int_0^1 \left( \sum_{i=1}^m \conj w_i 
    F_i(w,t \conj w) \right)dt. 
    \]
    Since each $F_i$ is a polynomial, 
    the function $G$ is also a polynomial. 
    A direct computation shows that $\bpa G =\alpha$.
\end{proof}

\begin{proof}[Proof of Theorem \ref{Thm:main_theorem}]

Let $\{Z_i\}_{i=1}^{n}$ be a basis of $\fg^{1,0}$ satisfying
(\ref{Equ:J_nilpotent}).
Using this basis, we take complex coordinates as in Subsection
\ref{subsection:Global complex coordinates}.
Then the $(0,1)$-vector fields
$\{ \widetilde{\conj{Z}}_i \}_{i=1}^n$ on $(G,J)$ can be expressed as in
(\ref{Equ:cZ_i}).

We now construct, by induction on $k=1,\ldots,n$, polynomials
\begin{equation}\label{Equ:w_k}
    w_k = z_k + H_k(z_1,\ldots,z_{k-1},\cz_1,\ldots,\cz_{k-1})
\end{equation}
such that
\begin{equation}\label{Equ:condition_of_w}
    \widetilde{\conj{Z}}_i (w_k) = 0
\end{equation}
for all $1 \le i \le n$.
Once this is done, 
a map 
\[
\Phi=(w_1,\ldots,w_n) \colon (G,J) \to \C^n
\]
is holomorphic, 
and the triangular form of $w_k$ in (\ref{Equ:w_k}) 
implies that its inverse is also a polynomial map.

For $k=1$, the choice $w_1=z_1$ satisfies the condition.
Assume that $w_1,\ldots,w_{k-1}$ have been constructed so as to satisfy the condition, and we construct $w_k$.
Since
$\{z_1,\ldots,z_{k-1},\cz_1,\ldots,\cz_{k-1}\}$ are polynomials in
$\{w_1, \ldots,w_{k-1},\cw_1, \ldots,\cw_{k-1}\}$,
it is enough to construct $H_k$ as a polynomial in
$\{w_1, \ldots,w_{k-1},\cw_1, \ldots,\cw_{k-1}\}$.
For $k \le i$, the condition (\ref{Equ:condition_of_w}) is automatically satisfied.
Thus it remains to construct $H_k$ so that 
(\ref{Equ:condition_of_w}) holds for all $i<k$.

For $1 \le i <k$, define $C_i$ and $D_i$ by
\[
    C_i = \sum_{j=k+1}^n \left(A_{ij} \pa_{z_j}
    + B_{ij} \pa_{\cz_j} \right), \quad
    D_i = \pa_{\cz_i} + \sum_{j=i+1}^{k-1} \left(A_{ij} \pa_{z_j}
    + B_{ij} \pa_{\cz_j} \right).
\]
Then we decompose $\widetilde{\conj{Z}}_i$ as
\[
\widetilde{\conj{Z}}_i = D_i + A_{ik} \pa_{z_k} + B_{ik} \pa_{\cz_k} + C_i.
\]
Since $C_i(H_k)=0$ and $\widetilde{\conj{Z}}_i(z_k)=A_{ik}$, we have
\begin{equation}\label{Equ:D(H)=-A}
\widetilde{\conj{Z}}_i(w_k)=0 \quad \Longleftrightarrow \quad
D_i(H_k) = -A_{ik}.
\end{equation}

We now express $D_i$ as a differential operator on $\C^{k-1}$ with coordinates
$w=(w_1,\ldots,w_{k-1})$.
By (\ref{Equ:w_k}), it can be written as
\[
    D_i = \pa_{\cw_i} + \sum_{j=i+1}^{k-1} \left(\tilde{A}_{ij} \pa_{w_j}
    + \tilde{B}_{ij} \pa_{\cw_j} \right),
\]
where $\tilde{A}_{ij}, \tilde{B}_{ij}$ are polynomials in
$\C[w,\cw]$.
Since
\[
D_i (w_j) = \widetilde{\conj{Z}}_i (w_j) = 0 \quad (1 \le j <k),
\]
we have $\tilde{A}_{ij}=0$.
Therefore, the vector fields $D_1,\ldots,D_{k-1}$ form a frame of $(0,1)$-vector fields in the $w$-coordinates.
It follows from this expression that we can take a $(0,1)$-form $\alpha_k$ with polynomial coefficients such that
\[
\alpha_k(D_i) = -A_{ik} \quad (1 \le i <k).
\]
We claim that $\bpa \alpha_k=0$.
Indeed, by the integrability and left-invariance of
$\{ \widetilde{\conj{Z}}_i \}_{i=1}^n$, for $1 \le l,m <k$ we can write
\begin{equation}\label{Equ:cZ_integrable}
[\widetilde{\cZ}_l, \widetilde{\cZ}_m]=
\sum_{r=1}^n c_{lm}^r \widetilde{\cZ}_r
\end{equation}
for some constants $c_{lm}^r \in \C$.
Applying this identity to $z_k$, we obtain
\[
\widetilde{\cZ}_l (A_{mk})-
\widetilde{\cZ}_m (A_{lk}) =
\sum_{r=1}^{k-1} c_{lm}^r A_{rk}.
\]
Taking the components of (\ref{Equ:cZ_integrable}) in the directions
$\pa_{z_1},\cdots,\pa_{z_{k-1}}, \pa_{\cz_1}, \ldots ,\pa_{\cz_{k-1}}$,
we get
\[
[D_l,D_m]=\sum_{r=1}^{k-1} c_{lm}^r D_r.
\]
Hence
\begin{align*}
    \bpa \alpha_k (D_l, D_m)
    &=-D_l A_{mk} + D_m A_{lk} - \alpha_k([D_l,D_m]) \\
    &=-\widetilde{\cZ}_l (A_{mk})+
    \widetilde{\cZ}_m (A_{lk}) +
    \sum_{r=1}^{k-1} c_{lm}^r A_{rk} \\
    &=0.
\end{align*}
Therefore, $\bpa \alpha_k=0$.
By Lemma \ref{Lem:Poincare}, there exists a polynomial $H_k$ in
$w_1, \ldots,w_{k-1},\cw_1, \ldots,\cw_{k-1}$ such that
$\bpa H_k = \alpha_k$.
For this choice of $H_k$, we have
\[
D_i(H_k)=\bpa H_k(D_i) = \alpha_k (D_i) = -A_{ik},
\]
and hence (\ref{Equ:D(H)=-A}) is satisfied.
Thus we have constructed $w_k$ satisfying (\ref{Equ:condition_of_w}).
This completes the proof.

\end{proof}

\section{Holomorphic polynomial crystallographic action}
\label{section:Holomorphic polynomial crystallographic action}

The preceding sections are concerned with a simply connected nilpotent Lie group
$G$ endowed with a left-invariant nilpotent complex structure $J$.  
In this section we record a consequence for lattices in $G$ 
and place it in the context of polynomial crystallographic actions 
in the sense of Dekimpe, Igodt and Lee. 
The point of view is the following: 
we fix a lattice $\Gamma\subset G$ and ask
whether $\Gamma$ can be realized as the fundamental group of a 
compact quotient of $\C^n$ by holomorphic polynomial automorphisms. 

We first recall the group-theoretic language used below.  Let $\Gamma$ be a
finitely generated nilpotent group.  Its \emph{Hirsch length} $h(\Gamma)$ is the
number of infinite cyclic factors in a polycyclic series for $\Gamma$, 
equivalently,
if
\[
  1=\Gamma_0\triangleleft \Gamma_1\triangleleft\cdots\triangleleft\Gamma_r=\Gamma
\]
is a subnormal series whose factors $\{\Gamma_i/\Gamma_{i-1}\}_{i=1}^r$
are finitely generated abelian groups, then
\[
  h(\Gamma) \coloneqq \sum_{i=1}^r \rank_{\Z}(\Gamma_i/\Gamma_{i-1}).
\]
This integer is independent of the chosen series.  We shall use the following
standard form of Malcev's theorem.

\begin{theorem}[\cite{Mal51}]
\label{thm:malcev}
Let $\Gamma$ be a torsion-free finitely generated nilpotent group.  Then there
exists a connected and simply connected nilpotent Lie group $G_\Gamma$ into
which $\Gamma$ embeds as a lattice.  Moreover, $G_\Gamma$ is unique up to a
unique Lie group isomorphism compatible with the embedding of $\Gamma$, and
\[
  \dim_{\R}G_\Gamma=h(\Gamma).
\]
\end{theorem}

We call $G_\Gamma$ the real Malcev completion of $\Gamma$. 
Throughout this
section, $\Gamma$ denotes a torsion-free finitely generated nilpotent group. 
We will mainly consider the even-dimensional case, namely the case where
\[
  h(\Gamma)=\dim_{\R}G_\Gamma=2n
\]
for some integer $n\ge 2$. 

\subsection{Real and holomorphic polynomial crystallographic action}

Let $P(\R^m)$ be the group of polynomial diffeomorphisms of $\R^m$, namely
polynomial maps whose inverse is again polynomial.  

\begin{definition}
\label{Def:real polynomial crystallographic action}
Let $\Gamma$ be a torsion-free finitely generated nilpotent group 
with $h(\Gamma)=m$. 
A \emph{real polynomial
crystallographic action} of $\Gamma$ is a homomorphism 
\[
  \rho:\Gamma \to P(\R^m)
\]
such that the induced action of $\Gamma$ on $\R^m$ is free, properly
discontinuous and cocompact. 
\end{definition}
This is the polynomial analogue of an affine crystallographic action, 
and was introduced and studied 
by Dekimpe, Igodt, and Lee in \cite{DIL96}. 
In this paper, 
we formulate a holomorphic analogue. 
Let $P_h(\C^n)$ be the group of holomorphic polynomial automorphisms of $\C^n$, 
namely holomorphic polynomial maps whose inverse is again polynomial. 

\begin{definition}
\label{Def:holomorphic polynomial crystallographic action}
Let $\Gamma$ be a torsion-free finitely generated nilpotent group with $h(\Gamma)=2n$. 
A \emph{holomorphic polynomial
crystallographic action} of $\Gamma$ is a homomorphism 
\[
  \rho:\Gamma \to P_h(\C^n)
\]
such that the induced action of $\Gamma$ on $\C^n$ is free, properly
discontinuous and cocompact. 
\end{definition}

Every holomorphic polynomial crystallographic action is, 
after identifying $\C^n$ with $\R^{2n}$, 
a real polynomial crystallographic action. 

\subsection{A lattice consequence of the main theorem}

We now formulate the consequence of the main theorem in terms of the group
$\Gamma$.

\begin{definition}
Let $\Gamma$ be a torsion-free finitely generated 
nilpotent group with $h(\Gamma)=2n$.
We say that $\Gamma$ is of \emph{nilpotent-complex type} if its Malcev
completion $G_\Gamma$ admits a left-invariant nilpotent complex structure.
\end{definition}

This condition is intrinsic to $\Gamma$, 
since $G_\Gamma$ is determined by $\Gamma$. 

\begin{theorem}
\label{Thm:main_theorem_lattice}
Let $\Gamma$ be a torsion-free finitely generated nilpotent group with
$h(\Gamma)=2n$.  If $\Gamma$ is of nilpotent-complex type, then $\Gamma$
admits a holomorphic polynomial
crystallographic action. 
\end{theorem}

\begin{proof}
Choose a left-invariant nilpotent complex structure $J$ on $G_\Gamma$.
By Theorem~\ref{Thm:main_theorem}, there exists a biholomorphism
\[
  \Phi:(G_\Gamma,J) \xrightarrow{\ \sim\ } \C^n
\]
which is polynomial with respect to the exponential coordinates on
$G_\Gamma$.

The left translations on $(G_\Gamma,J)$ are holomorphic.  Moreover, since
$G_\Gamma$ is nilpotent, the Baker--Campbell--Hausdorff formula shows that
the left translations are polynomial in exponential coordinates.  Therefore,
for every $g\in G_\Gamma$, we have
\[
  \Phi\circ L_g\circ \Phi^{-1}\in P_h(\C^n).
\]
Restricting this action to the lattice $\Gamma\subset G_\Gamma$, we obtain a
holomorphic polynomial crystallographic action
$\Gamma\to P_h(\C^n)$. 
\end{proof}

Thus the main result of this paper gives a positive answer to the holomorphic
polynomial crystallographic problem for the class of torsion-free finitely
generated nilpotent groups of nilpotent-complex type. 
It is then natural to ask whether there exists a torsion-free finitely
generated nilpotent group $\Gamma$ with $h(\Gamma)=2n$ which does not admit a
holomorphic polynomial crystallographic action.  At present, no such example
seems to be known.

\begin{question}\label{que:main_question}
    Does every torsion-free finitely generated nilpotent group
    $\Gamma$ with $h(\Gamma)=2n$ admit a holomorphic
    polynomial crystallographic action?
\end{question}

Suppose that $\Gamma$ admits a holomorphic polynomial crystallographic action.
Then it is natural to expect that $G_{\Gamma}$ 
admits a left-invariant complex structure $J$, 
and even that $(G_{\Gamma},J) \simeq \C^n$. 
If, however, 
one assumes more strongly that $\Gamma$ admits a holomorphic polynomial
crystallographic action \emph{of bounded degree}, 
then the remarkable uniqueness
theorem due to Benoist and Dekimpe 
imposes the expected restriction on $\Gamma$. 

Here, a (holomorphic) polynomial crystallographic action of 
$\Gamma$ is said to be \emph{of bounded degree} 
if the degrees of the polynomial transformations are bounded uniformly. 
Although the following theorem is stated for polycyclic-by-finite groups, 
which we do not define here, 
we note that every torsion-free finitely generated nilpotent group is 
polycyclic-by-finite.

\begin{theorem}[\cite{BD02}, Theorem 1]
\label{Thm:Benoist-Dekimpe}
    Let $\Gamma$ be a polycyclic-by-finite group, and let
    \[
        \rho_1,\rho_2:\Gamma\to P(\R^m)
    \]
    be two polynomial crystallographic actions of bounded degree.  Then
    $\rho_1$ and $\rho_2$ are polynomially conjugate.  In other words, there
    exists $\Psi \in P(\R^m)$ such that
    \[
        \Psi \circ \rho_1(\gamma)=\rho_2(\gamma)\circ \Psi
    \]
    for all $\gamma\in\Gamma$.
\end{theorem}

Using this theorem, we obtain the following: 

\begin{theorem}
\label{Thm:bounded-degree-necessary-condition}
Let $\Gamma$ be a torsion-free finitely generated nilpotent group with
$h(\Gamma)=2n$.  If $\Gamma$ admits a holomorphic polynomial
crystallographic action of bounded degree, then $G_{\Gamma}$ admits a
left-invariant complex structure $J$ such that
$(G_{\Gamma},J)\simeq \C^n$.
\end{theorem}

\begin{proof}
Let $\rho:\Gamma\to P_h(\C^n)$ 
be a holomorphic polynomial crystallographic action of bounded degree.
After identifying $\C^n$ with $\R^{2n}$, we regard $\rho$ as a real
polynomial crystallographic action of $\Gamma$ on $\R^{2n}$.
On the other hand, since $\Gamma$ is a lattice in its Malcev completion
$G_\Gamma$, the left action of $\Gamma$ on $G_\Gamma$ is properly
discontinuous and cocompact. 
Moreover, 
after choosing exponential
coordinates $\kappa \colon G_\Gamma \to \R^{2n}$, 
the Baker--Campbell--Hausdorff formula shows that
the left translations are polynomial maps.  Since $G_\Gamma$ is nilpotent,
their degrees are uniformly bounded.  Thus the canonical left action
\[
    \lambda:\Gamma\to P(\R^{2n}), \quad 
    \gamma \mapsto \kappa \circ L_{\gamma} \circ \kappa^{-1}
\]
is also a polynomial crystallographic action of bounded degree.
Applying Theorem~\ref{Thm:Benoist-Dekimpe} to the two actions $\lambda$ and
$\rho$, we obtain a polynomial diffeomorphism $\Psi \in P(\R^{2n})$ 
such that
\[
    \Psi \circ (\kappa \circ L_{\gamma} \circ \kappa^{-1}) \circ \Psi^{-1}=\rho(\gamma)
\]
for every $\gamma\in\Gamma$. 
Let $J$ be the pull-back of the standard complex structure on $\R^{2n} \simeq \C^n$ 
by $F \coloneqq \Psi \circ \kappa$. 
Then $F$ is a biholomorphism
\[
    F \, \colon \, (G_\Gamma,J) \xrightarrow{\ \sim\ } \C^n.
\]
Moreover,  
$L_{\gamma}$ is a biholomorphism of $(G_{\Gamma},J)$ for every $\gamma \in \Gamma$. 
It remains to see that $J$ is left-invariant under all of $G_\Gamma$, not only
under $\Gamma$. 

In exponential coordinates on $G_\Gamma$, the tensor
\[
    ((L_g)^*J-J)_x \in \Hom(T_x G_{\Gamma}, T_x G_{\Gamma})
\]
has coefficients which are polynomial functions of the point $x\in G_\Gamma$
and of the translating element $g\in G_\Gamma$.  Hence the condition
\[
    (L_g)^*J=J \quad \text{on }G_\Gamma
\]
is equivalent to the vanishing of finitely many polynomial functions of $g$. 

For every $\gamma\in\Gamma$, the map $L_\gamma$ is biholomorphic with respect
to $J$, and hence these polynomial equations are satisfied by all
$\gamma\in\Gamma$. 
Since $\Gamma$ is Zariski dense in $G_\Gamma$, 
the same equations hold for all $g\in G_\Gamma$. 
Therefore every left translation
$L_g$ preserves $J$, and so $J$ is left-invariant.
\end{proof}

Therefore, 
if one wants to find a group $\Gamma$ giving a negative answer to
Question \ref{que:main_question}, 
then a first promising candidate is a group
whose Malcev completion admits no left-invariant complex structure. 
Such examples already exist among $6$-dimensional nilpotent Lie groups.

Theorem \ref{Thm:bounded-degree-necessary-condition} would be particularly useful
if the bounded-degree assumption could be verified in the cases of interest. 
However, this seems to be a difficult problem. 
For instance, 
Dekimpe proved in \cite{Dek02} that polynomial crystallographic actions of
polycyclic-by-finite groups on $\R^2$ are of bounded degree, 
but such a general boundedness result is not known in higher dimensions. 
In particular, 
the theorem does not currently apply to the $6$-dimensional candidates 
mentioned above unless one can independently prove the bounded-degree property.

\section{Examples}
\label{section:Examples}

For a ring $R$, let $H_3(R)$ denote the Heisenberg group
\[
H_3(R)\coloneqq
\left\{
\begin{pmatrix}
1 & x & z \\
0 & 1 & y \\
0 & 0 & 1
\end{pmatrix}
\,\middle|\,
x,y,z\in R
\right\}.
\]

\begin{example}\label{Ex:2step}
    Let $K \coloneqq H_3(\R)\times \R$ 
    be the simply connected nilpotent Lie group whose Lie algebra is
    $\fk(\R) \coloneqq \fh_3 \times \R$. 
    We take a basis $\mathcal B \coloneqq (X_1,Y_1,X_2,Y_2)$
    of $\fk$ such that $[X_1,Y_1]=X_2$
    and all other brackets vanish.  Let
    \[
        \Gamma \coloneqq H_3(\Z)\times \Z
    \]
    be the lattice generated by 
    \[
        \gamma_1 \coloneqq \exp(X_1),\quad
        \gamma_2 \coloneqq \exp(Y_1),\quad
        \gamma_3 \coloneqq \exp(X_2),\quad
        \gamma_4 \coloneqq \exp(Y_2). 
    \]
    We now determine explicitly the holomorphic polynomial crystallographic
    action of $\Gamma$ on $\C^2$ obtained from the biholomorphism constructed
    above.

    Let $J$ be a nilpotent complex structure on $\fk$ such that 
    $Z_i \coloneqq (X_i - \iu Y_i)/2$ for $i=1,2$ define a basis of $\fk^{1,0}$. 
    We identify the Lie group $K$
    with $\C^2$ via global coordinates induced by the exponential map, 
    and the identification 
    $\psi_{\mathcal B} \colon \fk \xrightarrow{\ \sim\ } \R^4 \simeq \C^2$. 
    Now, since
    \[
        \left\{
        \begin{aligned}
        &\conj{Z}_1 + \frac{1}{2}[P,\conj{Z}_1]
        =\conj{Z}_1+\frac{\iu}{4}z_1(Z_2+\conj{Z}_2) \\
        &\conj{Z}_2 + \frac{1}{2}[P,\conj{Z}_2]
        =\conj{Z}_2, 
        \end{aligned}
        \right. 
    \]
    where $P=z_1Z_1+\conj{z}_1\conj{Z}_1+z_2Z_2+\conj{z}_2\conj{Z}_2$, 
    the left-invariant vector fields $\widetilde{\conj{Z}}_i$ on $K$
    are expressed in these coordinates as
    \[
        \left\{
        \begin{aligned}
        &\widetilde{\conj{Z}}_1=
        \pa_{\conj{z}_1}+\frac{\iu}{4}z_1 (\pa_{z_2}+\pa_{\conj{z}_2}) \\
        &\widetilde{\conj{Z}}_2=\pa_{\conj{z}_2}.  
        \end{aligned}
        \right. 
    \]
    Define complex-valued functions on $K$ by
    \[
        \left\{
        \begin{aligned}
        &w_1=z_1 \\
        &w_2=z_2-\frac{\iu}{4}z_1\cz_1 .
        \end{aligned}
        \right.
    \]
    Then $\widetilde{\conj{Z}}_i(w_j)=0$ for all $1 \le i,j \le 2$.
    Hence the map
    \[
        \Phi=(w_1,w_2) \colon (K,J) \to \C^2
    \]
    is holomorphic. Moreover, it is clearly bijective from its explicit form,
    and therefore $\Phi$ is biholomorphic.

    Using the coordinates $w=(w_1,w_2)$ on $K$,
    we compute $\Phi \circ L_{\gamma_i} \circ \Phi^{-1}$ for $i=1,2,3,4$,
    where $L_{\gamma_i}$ denotes the left translation by $\gamma_i$.  Then the
    induced holomorphic polynomial crystallographic action 
    of $\Gamma$ on $\C^2$ is given by 
    \[
        \left\{
        \begin{aligned}
        \gamma_1 \cdot(w_1,w_2)
        &=
        \left(
            w_1+1,\,
            w_2-\frac{\iu}{2}w_1-\frac{\iu}{4}
        \right),\\
        \gamma_2 \cdot(w_1,w_2)
        &=
        \left(
            w_1+\iu,\,
            w_2-\frac{1}{2}w_1-\frac{\iu}{4}
        \right),\\
        \gamma_3 \cdot(w_1,w_2)
        &=
        \left(
            w_1,\,
            w_2+1
        \right),\\
        \gamma_4 \cdot(w_1,w_2)
        &=
        \left(
            w_1,\,
            w_2+\iu
        \right).
        \end{aligned}
        \right.
    \]
\end{example}

\begin{example}\label{Ex:3step}
    Here we consider the Lie algebra $\fh_{16}$ 
    appearing in the classification of
    $6$-dimensional nilpotent Lie algebras 
    admitting complex structures \cite{COUV16}.
    Let $H_{16}$ be the simply connected nilpotent Lie group whose Lie algebra
    is the $6$-dimensional nilpotent Lie algebra $\fh_{16}$ with a basis
    $\mathcal{B} \coloneqq (X_1,Y_1,X_2,Y_2,X_3,Y_3)$
    such that
    \[
        \left\{
        \begin{aligned}
        &[X_1,Y_1]=Y_2, \\
        &[X_1,Y_2]=X_3, \\
        &[Y_1,Y_2]=Y_3,
        \end{aligned}
        \right.
    \]
    and all other brackets vanish. 
    Let $\Gamma$ be the lattice in $H_{16}$
    generated by
    \[
        \gamma_1 \coloneqq \exp(X_1),\;
        \gamma_2 \coloneqq \exp(Y_1),\;
        \gamma_3 \coloneqq \exp(X_2),\;
        \gamma_4 \coloneqq \exp(Y_2),\;
        \gamma_5 \coloneqq \exp(X_3),\;
        \gamma_6 \coloneqq \exp(Y_3).
    \]
    We now determine explicitly the holomorphic polynomial crystallographic
    action of $\Gamma$ on $\C^3$ obtained from the biholomorphism constructed
    above.

    Let $J$ be an almost complex structure on $\fh_{16}$ such that
    $Z_i \coloneqq (X_i - \iu Y_i)/2$ for $i=1,2,3$ define a basis of
    $(\fh_{16})^{1,0}$.  With respect to this basis, the Lie brackets are given by
    \[
        \left\{
        \begin{aligned}
        &[Z_1,\cZ_1]=-\frac{1}{2}Z_2+\frac{1}{2}\cZ_2, \\
        &[Z_1,Z_2]=-\frac{\iu}{2}Z_3, \\
        &[Z_1,\cZ_2]=\frac{\iu}{2}Z_3.
        \end{aligned}
        \right.
    \]
    From this, it follows that $J$ is integrable and nilpotent.
    We identify the Lie group $H_{16}$ with $\C^3$ via global coordinates
    induced by the exponential map, and the identification
    $\psi_{\mathcal B} \colon \fh_{16} \xrightarrow{\ \sim\ } \R^6 \simeq \C^3$. 
    Now, since
    \[
        \left\{
        \begin{aligned}
        &\cZ_1 + \frac{1}{2}[P,\cZ_1] + \frac{1}{12}[P,[P,\cZ_1]]
        =
        \cZ_1-\frac{1}{4}z_1Z_2+\frac{1}{4}z_1\cZ_2
        +\frac{\iu}{24}z_1^2Z_3
        +\left(
            \frac{\iu}{24}z_1\cz_1
            +\frac{\iu}{4}z_2
            -\frac{\iu}{4}\cz_2
        \right)\cZ_3,\\
        &\cZ_2 + \frac{1}{2}[P,\cZ_2] + \frac{1}{12}[P,[P,\cZ_2]]
        =
        \cZ_2+\frac{\iu}{4}z_1Z_3+\frac{\iu}{4}\cz_1\cZ_3,\\
        &\cZ_3 + \frac{1}{2}[P,\cZ_3] + \frac{1}{12}[P,[P,\cZ_3]]
        =
        \cZ_3,
        \end{aligned}
        \right.
    \]
    where 
    $P=
    z_1Z_1+\cz_1\cZ_1
    +z_2Z_2+\cz_2\cZ_2
    +z_3Z_3+\cz_3\cZ_3$, 
    the left-invariant vector fields $\widetilde{\cZ}_i$ on $H_{16}$ are
    expressed in these coordinates as
    \[
        \left\{
        \begin{aligned}
        \widetilde{\cZ}_1
        &=
        \pa_{\cz_1}
        -\frac{1}{4}z_1\pa_{z_2}
        +\frac{1}{4}z_1\pa_{\cz_2}
        +\frac{\iu}{24}z_1^2\pa_{z_3}
        +\left(
            \frac{\iu}{24}z_1\cz_1
            +\frac{\iu}{4}z_2
            -\frac{\iu}{4}\cz_2
        \right)\pa_{\cz_3},\\
        \widetilde{\cZ}_2
        &=
        \pa_{\cz_2}
        +\frac{\iu}{4}z_1\pa_{z_3}
        +\frac{\iu}{4}\cz_1\pa_{\cz_3},\\
        \widetilde{\cZ}_3
        &=
        \pa_{\cz_3}.
        \end{aligned}
        \right.
    \]
    Define complex-valued functions on $H_{16}$ by
    \[
        \left\{
        \begin{aligned}
        &w_1=z_1,\\
        &w_2=z_2+\frac{1}{4}z_1\cz_1,\\
        &w_3=z_3-\frac{\iu}{4}z_1\cz_2+\frac{\iu}{48}z_1^2\cz_1.
        \end{aligned}
        \right.
    \]
    Then $\widetilde{\cZ}_i(w_j)=0$ for all $1 \le i,j \le 3$.
    Hence the map
    \[
        \Phi=(w_1,w_2,w_3) \colon (H_{16},J) \to \C^3
    \]
    is holomorphic.  Moreover, it is clearly bijective from its explicit form,
    and therefore $\Phi$ is biholomorphic.

    Using the coordinates $w=(w_1,w_2,w_3)$ on $H_{16}$, 
    we compute $\Phi\circ L_{\gamma_i}\circ\Phi^{-1}$ for $i=1,\ldots,6$. 
    Then the
    induced holomorphic polynomial crystallographic action of $\Gamma$ on
    $\C^3$ is given by
    \[
        \left\{
        \begin{aligned}
        \gamma_1\cdot(w_1,w_2,w_3)
        &=
        \left(
            w_1+1,\,
            w_2+\frac{1}{2}w_1+\frac{1}{4},\,
            w_3+\frac{\iu}{8}w_1^2+\frac{\iu}{16}w_1
                -\frac{\iu}{4}w_2+\frac{\iu}{48}
        \right),\\
        \gamma_2\cdot(w_1,w_2,w_3)
        &=
        \left(
            w_1+\iu,\,
            w_2-\frac{\iu}{2}w_1+\frac{1}{4},\,
            w_3+\frac{1}{8}w_1^2+\frac{\iu}{16}w_1
                +\frac{1}{4}w_2-\frac{1}{48}
        \right),\\
        \gamma_3\cdot(w_1,w_2,w_3)
        &=
        \left(
            w_1,\,
            w_2+1,\,
            w_3-\frac{\iu}{4}w_1
        \right),\\
        \gamma_4\cdot(w_1,w_2,w_3)
        &=
        \left(
            w_1,\,
            w_2+\iu,\,
            w_3-\frac{3}{4}w_1
        \right),\\
        \gamma_5\cdot(w_1,w_2,w_3)
        &=
        \left(
            w_1,\,
            w_2,\,
            w_3+1
        \right),\\
        \gamma_6\cdot(w_1,w_2,w_3)
        &=
        \left(
            w_1,\,
            w_2,\,
            w_3+\iu
        \right).
        \end{aligned}
        \right.
    \]
\end{example}

\bibliographystyle{amsalpha}
\bibliography{crystal}
\end{document}